\documentclass[11pt,reqno]{article}
\usepackage[utf8]{inputenc}
\usepackage[english]{babel} 

\usepackage{amsmath}
\usepackage{amssymb}
\usepackage{amsthm}
\usepackage{amscd}
\usepackage[hyper]{amsbib}

\usepackage[matrix,arrow,curve]{xy}

\usepackage[a4paper, left=42mm, right=41mm, top=45mm, bottom=50mm]{geometry}

\begin{document}
\frenchspacing
\righthyphenmin=2

\renewcommand{\refname}{References}
\renewcommand{\bibname}{Bibliography}
\renewcommand{\proofname}{Proof}

\newtheorem{lm}{Lemma}
\newtheorem{lma}{Lemma}[lm]
\renewcommand{\thelma}{\arabic{lm}\Alph{lma}}
\newtheorem{tm}{Theorem}
\newtheorem*{ml}{Main lemma}
\newtheorem{cl}{Corollary}
\newtheorem*{cl*}{Corollary}
\theoremstyle{definition}
\newtheorem*{df}{Definition}
\newtheorem*{caa}{Constantly Acting Assumption}
\newtheorem*{caas}{Constantly Acting Assumptions}
\theoremstyle{remark}
\newtheorem*{rm*}{Remark}
\newtheorem{rmk}{Remark}
\newtheorem*{hint}{Hint}

\renewcommand{\Im}{\mathop{\mathrm{Im}}\nolimits}
\newcommand{\Dom}{\mathop{\mathrm{Dom}}\nolimits}
\newcommand{\Card}{\mathop{\mathrm{Card}}\nolimits}
\newcommand{\StU}{\mathop{\mathrm{StU}}\nolimits}
\newcommand{\Cent}{\mathop{\mathrm{Cent}}\nolimits}
\newcommand{\Ker}{\mathop{\mathrm{Ker}}\nolimits}
\newcommand{\inv}{^{-1}}
\newcommand{\M}{{\mathrm{M}}\,\!}
\newcommand{\E}{{\mathrm{E}}\,}
\newcommand{\GL}{\mathop{\mathrm{GL}}\nolimits}
\newcommand{\St}{\mathop{\mathrm{St}}\nolimits}
\newcommand{\Hf}{\mathop{\mathrm{H_1}}\nolimits}
\newcommand{\Hs}{\mathop{\mathrm{H_2}}\nolimits}
\newcommand{\GQ}{\mathop{\mathrm{GQ}}\nolimits}
\newcommand{\GH}{\mathop{\mathrm{GH}}\nolimits}
\renewcommand{\U}{\mathrm U\,\!}
\newcommand{\EU}{\mathop{\mathrm{EU}}\nolimits}
\newcommand{\KU}{\mathop{\mathrm{K_2U}}\nolimits}
\newcommand{\ind}{\mathop{\mathrm{ind}}\nolimits}

\newcommand{\sta}[1]{\StU(#1,\,R,\,\mathfrak{L})}
\newcommand{\st}{\sta{2n}}

\author{Andrei Lavrenov}
\title{On odd unitary Steinberg group}
\date{}
\maketitle

\begin{center} {\bf Abstract} \end{center}
Let $R$ be a ring with pseudo-involution, $\mathfrak L$ be an odd form parameter, $\U(2n,\,R,\,\mathfrak L)$ be an odd hyperbolic unitary group, $\EU(2n,\,R,\,\mathfrak L)$ be it elementary subgroup and $\StU(2n,\,R,\,\mathfrak L)$ be an odd unitary Steinberg group (see~\cite{VPoug,VP2}). We compute the Schur multipliers of $\StU(2n,\,R,\,\mathfrak L)$ and $\EU(R,\,\mathfrak L)$.\newline\newline
{\it 
Acknowledgement: the author gratefully acknowledges the support of the\newline
$1)$ State Financed research task $6.38.74.2011$ ``Structure theory and geometry of algebraic groups and their applications in representation theory and algebraic $\mathrm K$-theory'' at the SPbSU,\newline
$2)$ RFBR project $13$-$01$-$00709$ ``Study of algebraic groups over rings by localization methods'',\newline
$3)$ Chebyshev Laboratory (Department of Mathematics and Mechanics, SPbSU) under RF Government grant $11.\mathrm G34.31.0026$.
} 

\section*{Introduction}

Let $\st$ denote the analog of the Steinberg group for an odd unitary group defined by Petrov in~\cite{VPoug}. The main goal of the present paper is to prove that $\st$ is centrally closed.

Let $R$ be an associative ring with 1, $n\geq3$. Denote by $\E(n,\,R)$ the usual elementary group, i.e. the subgroup of the group of invertible $n\times n$ matrices, generated by elementary transvections $t_{ij}(a)$. It is obviously that Steinberg relations
\setcounter{equation}{0}
\renewcommand{\theequation}{S\arabic{equation}}
\begin{align}
&t_{ij}(a)t_{ij}(b)=t_{ij}(a+b),\\
&[t_{ij}(a),\,t_{kh}(b)]=1\ \text{ for }k\neq j,\ h\neq i,\\
&[t_{ij}(a),\,t_{jk}(b)]=t_{ik}(ab)
\end{align}
hold for transvections. The abstract group defined by generators $\{x_{ij}(a)\mid 1\leq i\neq j\leq n,\ a\in R\}$ and relations S1--S3 with $x_{ij}(a)$ instead of $t_{ij}(a)$ is called the (linear) Steinberg group $\St(n,\,R)$. It follows from S3 that the commutator subgroup $[\St(n,\,R),\,\St(n,\,R)]$ of the Steinberg group coincides with it and thus $\Hf(\St(n,\,R),\,\mathbb{Z})$ is trivial. For such groups (called perfect) the kernel of the so called universal central extension coincides with the second homology group $\Hs(\St(n,\,R),\,\mathbb{Z})$ and in this context is often called the Schur multiplier (see~\cite{JMb}, \cite{RSb}). It is known that for $n\geq5$ the Schur multiplier of $\St(n,\,R)$ is trivial, i.e. this group is centrally closed (or superperfect), see~\cite{JMb}, \cite{RSb}. 

Parallel results are known in the cases of other classical\footnote{One can also find parallel results for exceptional Chevalley groups in~\cite{MSgrc,vdKSosm, S2}, but they are not generalized in the present paper.} Chevalley groups (see~\cite{MSgrc, vdKSosm, vdK2, S2}), Bak's quadratic groups $\GQ(2n,\,R,\,\Lambda)$ (see~\cite{ABb, AB2, AB3, ABNV, Dibd, HOb, Haz, HV, HVZh, StepV}) and Hermitian groups $\GH(2n,\,R,\,a_1,\ldots,a_r)$ (see~\cite{GThgk,BTshk}). In the present paper we show that the Schur multiplier of Petrov's odd unitary Steinberg group is trivial when $n\geq5$, which is a common generalization of all above results. The condition $n\geq5$ can not be relaxed, since the orthogonal Steinberg group which is the special case of an odd unitary Steinberg group is not necessary centrally closed for $n=4$ (see~\cite{vdKSosm}). Using these results one can obtain that ``in the limit'' the Schur multiplier of elementary odd unitary group $\EU(R,\,\mathfrak L)$ coincides with the kernel $\KU(R,\,\mathfrak L)$ of the natural epimorphism $\StU(R,\,\mathfrak L)\twoheadrightarrow\EU(R,\,\mathfrak L)$:
$$
\KU(R,\,\mathfrak L)=\Hs(\EU(R,\,\mathfrak L),\,\mathbb{Z}).
$$

The paper is organized as follows. In Section~1 we recall the definition of the Schur multiplier and prove several elementary facts about it, in Section~2 we define an odd unitary Steinberg group following~\cite{VPoug}, in Section~3 we prove the main technical lemma and in the last section we obtain the main results.

I would like to express my thanks to Professor Nikolai Vavilov for his supervision of this work and Dr. Sergey Sinchuk for helpful discussions.

\section{Central extensions}
\begin{df}
An epimorphism of abstract groups $\epsilon:H\twoheadrightarrow G$ is called {\it a central extension (of $G$)} if it its kernel is contained in the center of $H$.
\end{df}

\begin{lm}[Steinberg central trick]
\label{sct}
Let $\epsilon:H\twoheadrightarrow G$ be a central extension. Then for any elements $u_1$, $u_2$, $v_1$, $v_2\in H$ such that $\epsilon(u_1)=\epsilon(u_2)$ and $\epsilon(v_1)=\epsilon(v_2)$ one has $[u_1,\,v_1]=[u_2,\,v_2]$.
\end{lm}

\begin{hint}
Use that $u_1{u_2}\inv$, $v_1{v_2}\inv\in\Ker(\epsilon)\subseteq\Cent(H)$.
\end{hint}

\newcommand{\epi}{\twoheadrightarrow}

\begin{df}
Let $\epsilon:H\epi G$ be a central extension. Then for $x$, $y\in G$ denote by $[\epsilon\inv x,\,\epsilon\inv y]$ the commutator of any $u$, $v\in H$ such that $\epsilon(u)=x$ and $\epsilon(v)=y$. This definition is correct by Lemma~\ref{sct}.
\end{df}

\begin{df}
Group $G$ is called {\it perfect} if it coincides with its commutator subgroup $[G,\,G]$.
\end{df}

\begin{lm}
\label{unique}
Let $\epsilon:H\epi G$ be a central extension. Then $H$ is perfect if and only if for any central extension $\zeta:H'\epi G$ there exists no more than one homomorphism making the diagram
$$
\xymatrix{
H\ar@{-->}[rr]\ar@{->>}[rd]_\epsilon&&H'\ar@{->>}[ld]^\zeta\\
&G
}
$$
commutative.
\end{lm}

\begin{proof}
Let $H$ be a perfect group, $\zeta$ be a central extension of $G$ and $\eta$, $\theta$ be homomorphisms making the diagram
$$
\xymatrix{
H\ar@<0.5ex>[rr]^\eta\ar@<-0.5ex>[rr]_\theta\ar@{->>}[rd]_\epsilon&&H'\ar@{->>}[ld]^\zeta\\
&G
}
$$
commutative. Then for any two elements $x$, $y$ from $H$ by lemma~\ref{sct} we have that $[\eta(x),\,\eta(y)]=[\zeta\inv(\epsilon(x)),\,\zeta\inv(\epsilon(y))]=[\theta(x),\,\theta(y)]$, i.e. $\eta([x,\,y])=\theta([x,\,y])$. But $H$ is perfect, so $\eta=\theta$.

Now suppose that $H$ is not perfect. Then there is a nontrivial ho\-mo\-mor\-phism $\alpha:H\rightarrow A=\dfrac{H}{[H,\,H]}$ onto abelian group. Using the fact that the morphism $\zeta:H\times A\rightarrow G$ defined by $\zeta(u,a)=\epsilon(u)$ is a central extension, we obtain distinct homomorphisms $\eta(u)=(u,1)$ and $\theta(u)=(u,\alpha(u))$, making the diagram
$$
\xymatrix{
H\ar@<0.5ex>[rr]^\eta\ar@<-0.5ex>[rr]_\theta\ar@{->>}[rd]_\epsilon&&H\times A\ar@{->>}[ld]^\zeta\\
&G
}
$$
commutative.
\end{proof}

\begin{df}
A central extension $\pi:U\epi G$ is called {\it a universal central extension of $G$} if for any central extension $\epsilon:H\epi G$ there is a unique group homomorphism $\eta$ making the diagram 
$$
\xymatrix{
U\ar@{-->}[rr]^\eta\ar@{->>}[rd]_\pi&&H\ar@{->>}[ld]^\epsilon\\
&G
}
$$
commutative. 
\end{df}

\begin{rm*}
Lemma~\ref{unique} implies that the domain (and thus the image) of a universal central extension is perfect. So the universal central extension can only exist for a perfect group.
\end{rm*}

\begin{lm}
\label{uceexists}
Any perfect group admits a universal central extension.
\end{lm}

\begin{proof}
Fix an epimorphism $\phi:F\epi G$ with $F$ free and $R=\Ker\phi$. Obviously $\phi$ factors through the factorization epimorphism $\tau:F\epi \overline F=\dfrac F{[R,\,F]}$.
$$
\xymatrix{
F\ar@{->>}[r]^{\phi}\ar@{->>}[d]_\tau&G\\
\overline F\ar@{-->}[ru]_{\varphi}
}
$$
Restricting the factored morphism $\varphi$ to the commutator subgroup we obtain the morphism $\pi:\dfrac{[F,\,F]}{[R,\,F]}\epi[G,\,G]=G$ with $\Ker\pi=\dfrac{R\cap[F,\,F]}{[R,\,F]}$. Obviously $\varphi$ and $\pi$ are central extensions of $G$ and below we will show that $\pi$ is a universal central extension.

Fix a central extension $\epsilon:H\epi G$. Applying the universal property of $F$ one can obtain a group morphism $\theta:F\rightarrow H$ making the diagram
$$
\xymatrix{
F\ar@{-->}[rr]^\theta\ar@{->>}[rd]_\phi&&H\ar@{->>}[ld]^\epsilon\\
&G
}
$$
commutative. Thus $[\theta(R),\,\theta(F)]\subseteq[\Ker\epsilon,\,\theta(F)]=1$ so $\theta$ factors through $\tau$.
$$
\xymatrix{
F\ar@{->}[rr]^\theta\ar@{->>}[rd]^\tau\ar@{->>}@/_/[rdd]_\phi&&H\ar@{->>}@/^/[ldd]^\epsilon\\
&\overline F\ar@{-->}[ru]^\vartheta\ar@{->>}[d]|{\varphi\ \ }&\\
&G&
}
$$
Restricting the factored morphism $\vartheta$ to the commutator subgroup we obtain morphism $\eta:[\overline F,\,\overline F]\rightarrow H$ making the diagram 
$$
\xymatrix{
[\overline F,\,\overline F]\ar@{-->}[rr]^\eta\ar@{->>}[rd]_\pi&&H\ar@{->>}[ld]^\epsilon\\
&G
}
$$
commutative. Finally, observe that for any $x$, $y\in\overline F$ there are $u$, $v\in[\overline F,\,\overline F]$ such that $\varphi(x)=\pi(u)$ and $\varphi(y)=\pi(v)$ (since $\pi$ is a surjection) but $\varphi$ is a central extension so that it follows from Lemma~\ref{sct} that $[x,\,y]=[u,\,v]$ and thus $[\overline F,\,\overline F]$ is a perfect group. Now one can use Lemma~\ref{unique} to finish the proof.
\end{proof}

\begin{df}
If $\pi$ is a universal central extension of $G$ then its kernel is called {\it the Schur multiplier of $G$} and denoted by $\M(G)=\Ker(\pi)$. This definition is obviously correct, i.e. if $\pi$ and $\varpi$ are two universal central extensions of $G$, their kernels are isomorphic, $\Ker(\pi)\cong\Ker(\varpi)$.
\end{df}

\begin{rm*}
Lemma~\ref{uceexists} implies that the Schur multiplier is defined for all perfect groups (and only for them).
\end{rm*}

\begin{df}
A perfect group $G$ is called {\it centrally closed} if $1_G:G\rightarrow G$ is the universal central extension. Obviously this is equivalent to $\M(G)=1$.
\end{df}

\begin{df}
A central extension $\epsilon:H\epi G$ is called {\it split} if there is a homomorphism $\sigma:G\rightarrow H$ such that $\epsilon\sigma=1_G$, i.e. the short exact sequence
$$
\begin{CD}
1@>>>\Ker\epsilon@>>>H@>\epsilon>>G@>>>1
\end{CD}
$$
is (right) splitting.
\end{df}

\begin{lm}
\label{perfectsplit}
Let $\epsilon:H\epi G$ be a split central extension with $H$ perfect. Then $H$ and $G$ are isomorphic.
\end{lm}
\begin{proof}
Let $\sigma:G\rightarrow H$ be a homomorphism such that $\epsilon\sigma=1_G$. Then we have a commutative diagram
$$
\xymatrix{
H\ar@<0.5ex>[rr]^{\sigma\epsilon}\ar@<-0.5ex>[rr]_{1_H}\ar@{->>}[rd]_\epsilon&&H\ar@{->>}[ld]^\epsilon\\
&G
}
$$
with $H$ perfect. By Lemma~\ref{unique} it follows that $\sigma\epsilon=1_H$ and thus $D\cong G$.
\end{proof}

\begin{lm}
\label{ccisuce}
Let $\epsilon:H\epi G$ be a central extension with $H$ centrally closed. Then $\epsilon$ is the universal central extension of $G$.
\end{lm}
\begin{proof}
$G=\epsilon(H)$ is perfect so that one can consider the universal central extension $\pi:U\epi G$ and a group homomorphism $\eta:U\rightarrow H$ making the diagram
$$
\xymatrix{
U\ar[rr]^{\eta}\ar@{->>}[rd]_\pi&&H\ar@{->>}[ld]^\epsilon\\
&G
}
$$
commutative. Thus by Lemma~\ref{sct} $\eta(U)=[H,\,H]=H$ so $\eta:U\epi H$ is a central extension. But $H$ is centrally closed so using the universal property of $1_H$ we obtain that $\eta$ splits. Now use Lemma~\ref{perfectsplit} to finish the proof.
\end{proof}

\section{Odd unitary Steinberg group}

\begin{df}
Let $R$ be an associative ring with identity. An additive map $\overline{\phantom a}:R\rightarrow R$ such that $\overline1$ is invertible, $\overline{\overline a}=a$ and $\overline{ab}=\overline b\,{\overline1}\inv\overline a$ for any $a$, $b\in R$ is called {\it a pseudo-involution}. In this paper $R$ will always denote an associative ring with identity and pseudo-involution on it.
\end{df}

\begin{df}
A biadditive map $B:V_R\times V_R\rightarrow R$ is called {\it an anti-Hermitian form (on $V_R$)} if it satisfies the following axioms 
\begin{align*}
&1)\,B(ua,\,vb)=\overline{a}\overline{1}\inv B(u,\,v)b,\\
&2)\,B(u,\,v)=-\overline{B(v,\,u)}
\end{align*}
for any $u$, $v\in V$ and $a$, $b\in R$.
\end{df}

\begin{df}
Let $B$ be an anti-Hermitian form on $V_R$. Then the set $V\times R$ with the composition law given by
$$
(u,\,a)\dotplus(v,\,b)=(u+v,\,a+b+B(u,\,v))
$$
is called {\it the Heisenberg group $\mathfrak{H}$ of the form $B$}. 
\end{df}

\begin{rm*}
It is clear that $\dotplus$ is associative, $(0,\,0)$ is the identity element and the inverse is given by
$$
\dot-(u,\,a)=(-u,\,-a+B(u,u)),
$$
so Heisenberg group is actually a group.
\end{rm*}

\begin{df}
Let $B$ be an anti-Hermitian form on $V_R$ and $\mathfrak{H}$ be its Heisenberg group. We can define the right action of $R$ on $\mathfrak{H}$ by
$$
(u,\,a)\leftharpoonup b=(ub,\,\overline{b}\,\overline{1}\inv ab).
$$
\end{df}

\begin{rm*}
It's easy to see that
$$
\lambda\leftharpoonup a\leftharpoonup b=\lambda\leftharpoonup ab
$$
and
$$
(\lambda\dotplus\mu)\leftharpoonup a=\lambda\leftharpoonup a\dotplus\mu\leftharpoonup a
$$
for all $\lambda$, $\mu\in\mathfrak{H}$ and $a$, $b\in R$.
\end{rm*}

\begin{df}
Subgroups of a Heisenberg group $\mathfrak H$
$$
\mathfrak L_{\min}=\{(0,\,a+\overline a)\mid a\in R\}\quad\text{and}\quad\mathfrak L_{\max}=\{(u,\,a)\mid a=\overline a+B(u,\,u)\}
$$
are called {\it the minimal} and {\it the maximal odd form parameters}, respectively.
\end{df}

\begin{rm*}
It's clear that $\mathfrak L_{\min}\leq\mathfrak L_{\max}$ and that $\mathfrak L_{\min}$ and $\mathfrak L_{\max}$ are stable under the action of $R$.
\end{rm*}

\newcommand{\fp}{\mathfrak L}

\begin{df}
A subgroup $\fp$ of a Heisenberg group $\mathfrak H$ is called {\it an (odd) form {pa\-ra\-me\-ter}} if $\mathfrak L_{\min}\leq\fp\leq\mathfrak L_{\max}$ and $\fp$ is stable under the action of $R$.

The triple $(V,\,B,\,\fp)$ is called {\it an odd quadratic space}.

The orthogonal sum of two odd quadratic spaces $V$ and $V'$ is constructed as follows: the underlying module is $V\oplus V'$, the anti-Hermitian form is given by $(B+B')(u+u',\,v+v')=B(u,\,v)+B'(u',\,v')$ and the odd form parameter consists of all pairs $(u+u',\,a+a')$, where $(u,\,a)\in\fp$ and $(u',\,a')\in\fp'$.
\end{df}

\begin{df}
Let $V'$ and $V$ be two modules over $R$ with bilinear forms $B'$ and $B$. A module homomorphism $f:V'\rightarrow V$ is called an {\it isometry} if $B(fu,\,fv)=B'(u,\,v)$ for all $u$, $v\in V'$.

If $\fp$ is an odd form parameter for $V$ and $f$ and $g$ are isometries from $V'$ to $V$ such that $(fv-gv,\,B(gv-fv,gv))\in\fp$ for every $v\in V$ we say that $f$ and $g$ are equivalent modulo $\fp$ and write $f\equiv g\mod\fp$. One can see that it is an equivalence relation between the isometries.

The {\it odd unitary group} $\U(V,\,B,\,\fp)$ of the odd quadratic space $V$ is the group of all bijective isometries of V onto itself that are equivalent to the identity map modulo $\fp$.
\end{df}

\begin{df}
Consider a free module $H$ spanned on vectors $e_1$, $e_{-1}$ and anti-Hermitian form $B$ on it such that $B(e_1,\,e_{-1})=1$, $B(e_1,\,e_1)=B(e_{-1},\,e_{-1})=0$. Denote $\fp=\{(e_1a+e_{-1}b,\,\overline a{\overline1}\inv b+c+\overline c)\mid a,\,b,\,c\in R\}$. One can check that $\fp$ is a form parameter for $H$.

Denote by $H^n$ the orthogonal sum of $n$ copies of $H$. Its basis coming from the bases of the summands will be indexed as follows: $e_1,\ldots,e_n,e_{-n},\ldots,e_{-1}$.

Suppose we are given an odd quadratic space $(V,\,B,\,\fp)$. The orthogonal sum $H^n\oplus V$ is called an {\it odd hyperbolic unitary space} of rank $n$. The unitary group of $H^n\oplus V$ is called the {\it odd hyperbolic unitary group} and denoted $\U(2n,\,R,\,\fp)$.

Consider an odd quadratic space $(V,\,B,\,\fp)$. A pair of vectors $(u,\,v)$ from $V$ such that $B(u,\,v)=1$, $(u,\,0)$, $(v,\,0)\in\fp$ is called a {\it hyperbolic pair}. The greatest $n$ satisfying the condition that there exist $n$ mutually orthogonal hyperbolic pairs in $V$ is called the {\it Witt index} of $V$ and denoted by $\ind(V,\,B,\,\fp)$. It is easy to see that the Witt index coincides with the greatest $n$ satisfying the condition that there exists an isometry $f$ of the space $H^n$ to $V$ such that $(fu,\,a)\in\fp$ for $(u,\,a)$ from the form parameter of $H^n$.

Suppose that the Witt index of $V$ is at least $n$. Fix an embedding of $H^n$ to $V$, i.e. fix elements $e_1,\ldots,e_n,e_{-n},\ldots,e_{-1}$ in $V$ such that $(e_i,\,e_j)=0$ for $i\neq j$, $(e_i,\,e_{-i})=1$ for $i\in\{1,\ldots,n\}$, $(e_i,\,0)\in\fp$. Define $V_0$ as the orthogonal complement to $\sum_{i=1}^{-1}e_iR$ in $V$ (it can be defined since the restriction of $B$ to this subspace is nonsingular), $B_0$ as the restriction of $B$ to $V_0$ and $\fp_0$ as the restriction of $\fp$ to $V_0$. Then it is easy to see that $V$ is isometric to the odd hyperbolic space $H^n\oplus V_0$. Thus the unitary group of an odd quadratic space with Witt index at least $n$ can be identified with the odd hyperbolic unitary group $\U(2n,\,R,\,\fp)$ corresponding to an appropriate odd form parameter. 
\end{df}

\begin{df}
Let $\U(2n,\,R,\,\fp)$ be an unitary group of odd hyperbolic space $V$, denote by $\Omega_+$, $\Omega_{-}$ the sets $\{1,\ldots,n\}$, $\{-n,\ldots,-1\}$ respectively, set $\Omega=\Omega_+\cup\Omega_-$. Set $\varepsilon_i=\overline1\inv$ if $i\in\Omega_+$ and $\varepsilon_i=-1$ if $i\in\Omega_-$. 

For $i\in\Omega$, $j\in\Omega\setminus\{\pm j\}$, $a\in R$, $(u,\,b)$ denote by $T_{ij}(a)$ linear trans\-for\-mation of $V$ to itself
$$
T_{ij}(a):\ w\mapsto w+e_{-j}\varepsilon_{-j}\overline a{\overline1}\inv(e_i,\,w)-e_{i}a\varepsilon_j(e_{-j},\,w)
$$
and by $T_i(u,\,b)$ transformation 
$$
T_i(u,\,b):\ w\mapsto w-e_i\varepsilon_i(u,\,w)-e_i\varepsilon_{i}b\varepsilon_{-i}(e_i,\,w)+u\varepsilon_{-i}(e_i,\,w). 
$$
The transformations $T_{ij}(a)$ and $T_i(u,\,b)$ are called {\it(odd unitary) elementary trans\-vec\-tions}. One can check that elementary transvections lie in $\U(2n,\,R,\,\fp)$ (see~\cite{VPoug}). The subgroup of the hyperbolic unitary group generated by elementary trans\-vec\-tions is called an {\it odd hyperbolic elementary group} and denoted $\EU(2n,\,R,\,\fp)$.
\end{df}

The linear Steinberg group is defined by ``elementary'' relations between linear transvections. Now we will define the unitary Steinberg group by the relations between unitary transvections.

\begin{df}
Let $n\geq3$. The odd unitary Steinberg group $\StU(2n,\,R,\,\fp)$ is the group defined by generators $\{X_{ij}(a)\mid i,\,j\in\Omega,\ i\not\in\{\pm j\},\ a\in R\}\cup\{X_i(\xi)\mid i\in\Omega,\ \xi\in\mathfrak{L}\}$ and relations

\setcounter{equation}{-1}
\renewcommand{\theequation}{R\arabic{equation}}
\begin{align}
&X_{ij}(a)=X_{-j,-i}(\varepsilon_{-j}\overline{a}\varepsilon_i),\\
&X_{ij}(a)X_{ij}(b)=X_{ij}(a+b),\\
&X_i(\xi)X_i(\zeta)=X_i(\xi\dotplus\zeta),\\
&[X_{ij}(a),\,X_{hk}(b)]=1,\text{ for }h\not\in\{j,-i\},\ k\not\in\{i,-j\},\\
&[X_i(\xi),\,X_{jk}(a)]=1,\text{ for }j\neq-i,\ k\neq i,\\
&[X_{ij}(a),\,X_{jk}(b)]=X_{ik}(ab),\\
&[X_i(u,a),\,X_j(v,b)]=X_{i,-j}(\varepsilon_iB(u,\,v)),\text{ for }i\not\in\{\pm j\},\\
&[X_i(u,a),\,X_i(v,b)]=X_i(0,B(u,\,v)-B(v,\,u)),\\
&[X_i(u,\,a),\,X_{-i,j}(b)]=X_{ij}(\varepsilon_iab)X_{-j}((u,-\overline{a})\leftharpoonup b),\\
&[X_{ij}(a),\,X_{j,-i}(b)]=X_i(0,-\varepsilon_{-i}\overline{1}ab+\overline{b}\,\overline{1}^{-1}\overline{a}\varepsilon_i),
\end{align}
where commutators are left-normed.
\end{df}

\begin{rmk}
Relation~R1 implies that ${X_{ij}(0)}^2=X_{ij}(0)$, i.e. $X_{ij}(0)=1$ and thus that $X_{ij}(-a)=X_{ij}(-a)X_{ij}(a){X_{ij}(a)}\inv={X_{ij}(a)}\inv$. Similarly, R2 implies that $X_i(0,\,0)=1$ and $X_i(\dot-(u,\,a))={X_i(u,\,a)}\inv$.
\end{rmk}

\begin{rmk}
Relation~R7 is the direct consequence of the relation~R2. It is listed to emphasize that $X_i(u,\,a)$ and $X_i(v,\,b)$ do not commute in general.
\end{rmk}

\setcounter{rmk}{0}

One can check the following result.

\begin{lm}
Relations {\rm R0--R9} hold for elementary transvections $T_{ij}(a)$ and $T_i(u,\,a)$. Thus when $n\geq3$ there is a natural epimorphism from $\st$ to $\EU(2n,\,R,\,\fp)$ sending the generators of the Steinberg group to the cores\-pon\-ding elementary transvections.
\end{lm}

\begin{df}
Define $^{\StU}\!\U_1(2n,\,R,\,\fp)$ to be a subgroup of $\st$ ge\-ne\-ra\-ted by $\big\{X_{n,i}(a)\mid i\in\Omega\setminus\{\pm n\},\ a\in R\big\}\cup\big\{X_n(\zeta)\mid\zeta\in\fp\big\}$ and $^{\EU}\!\U_1(2n,\,R,\,\fp)$ to be its image in the $\EU(2n,\,R,\,\fp)$.
\end{df}

\begin{lm}
\label{ur}
$^{\StU}\!\U_1(2n,\,R,\,\fp)\cong\,^{\EU}\!\U_1(2n,\,R,\,\fp)$.
\end{lm}
\begin{proof}
First observe that $[X_n(\zeta),\,X_{n,i}(a)]=1=[X_{n,i}(a),\,X_{n,j}(b)]$ for $i\in\Omega\setminus\{\pm n\}$, $j\in\Omega\setminus\{\pm n,\pm i\}$ and $[X_{n,i}(a),\,X_{n,-i}(b)]=X_n(\xi)$ for some $\xi\in\fp$ so that any $x\in\,^{\StU}\!\U_1(2n,\,R,\,\fp)$ can be decomposed as 
$$x=X_n(\zeta)\cdot X_{n,1}(a_1)\cdot X_{n,2}(a_2)\cdot\ldots\cdot X_{n,-1}(a_{-1}).$$ 
Now we will check that this decomposition is unique. Let 
$$
X_n(\zeta)\cdot X_{n,1}(a_1)\cdot X_{n,2}(a_2)\cdot\ldots\cdot X_{n,-1}(a_{-1})=X_n(\xi)\cdot X_{n,1}(b_1)\cdot X_{n,2}(b_2)\cdot\ldots\cdot X_{n,-1}(b_{-1}).
$$ 
Then 
\begin{multline*}
1=X_n(\zeta)\cdot X_{n,1}(a_1)\cdot\ldots\cdot X_{n,-1}(a_{-1})\cdot X_{n,-1}(-b_{-1})\cdot\ldots\cdot X_{n,1}(-b_1)\cdot X_{n}(\xi)=\\=X_n(\eta)\cdot X_{n,1}(a_1-b_1)\cdot\ldots\cdot X_{n,-1}(a_{-1}-b_{-1}).
\end{multline*}
Then also $T_n(\eta)\cdot T_{n,1}(a_1-b_1)\cdot\ldots\cdot T_{n,-1}(a_{-1}-b_{-1})=1$ thus $a_i=b_i$ for all $i\in\Omega\setminus\{\pm n\}$ and thus $\zeta=\xi$. Now the claimed result is obvious.
\end{proof}

\begin{df}
One can define 
$$
^{\StU}\!\U_1^-(2n,\,R,\,\fp)=\big\langle X_{-n,i}(a),\ X_{-n}(\zeta)\mid i\in\Omega\setminus\{\pm n\},\ a\in R,\ \zeta\in\fp\big\rangle,
$$ 
its image $^{\EU}\!\U_1^-(2n,\,R,\,\fp)$ and check that they are isomorphic.
\end{df}

\begin{lm}
\label{urgen}
Steinberg group $\st$ is generated by $^{\StU}\!\U_1(2n,\,R,\,\fp)$ and $^{\StU}\!\U_1^-(2n,\,R,\,\fp)$.
\end{lm}
\begin{hint}
Use R5 and R8.
\end{hint}

\begin{df}
Define $\KU(2n,\,R,\,\fp)$ to be the kernel of natural epimorphism of the Steinberg group $\st$ onto $\EU(2n,\,R,\,\fp)$.
$$
\KU(2n,\,R,\,\fp)\rightarrowtail\st\epi\EU(2n,\,R,\,\fp)
$$
\end{df}

\begin{lm}
\label{cent}
Consider a natural mapping $\phi_n:\st\rightarrow\StU(2n+2,\,R,\,\fp)$, sending $X_{ij}(a)$ to $X_{ij}(a)$ and $X_i(\zeta)$ to $X_i(\zeta)$. Then $$\phi_n(\KU(2n,\,R,\,\fp))\subseteq\Cent(\StU(2n+2,\,R,\,\fp)).$$
\end{lm}
\begin{proof}
Fix $x\in\KU(2n,\,R,\,\fp)$ and $y\in\,^{\StU}\!\U_1(2n+2,\,R,\,\fp)$. Steinberg relations imply that $\phi_n(x)\cdot y\cdot\phi_n(x)\inv\in\,^{\StU}\!\U_1(2n+2,\,R,\,\fp)$. But $\phi_n(x)\in\KU(2n+2,\,R,\,\fp)$ so images of $\phi_n(x)\cdot y\cdot\phi_n(x)\inv$ and $y$ coincide in $^{\EU}\!\U_1(2n+2,\,R,\,\fp)$ and thus by lemma~\ref{ur} $\phi_n(x)\cdot y\cdot\phi_n(x)\inv=y$. Similarly, for any $z\in\,^{\StU}\!\U_1^-(2n+2,\,R,\,\fp)$ one has $[\phi_n(x),\,z]=1$. Now use Lemma~\ref{urgen}.
\end{proof}

\begin{rm*}
Centrality of $\KU(2n,\,R,\,\fp)$ in $\st$ is not so easy to obtain (see~\cite{vdK,Tul} for the linear case).
\end{rm*}

\begin{df}
Define $\StU(\infty,\,R,\,\fp)=\StU(R,\,\fp)$, $\EU(R,\,\fp)$ and $\KU(R,\,\fp)$ as direct limits of corresponding sequences.
$$
\xymatrix{
\ldots\ar[r]&\KU(2n,\,R,\,\fp)\ar[r]\ar@{>->}[d]&\KU(2n+2,\,R,\,\fp)\ar[r]\ar@{>->}[d]&\ldots\\
\ldots\ar[r]&\StU(2n,\,R,\,\fp)\ar[r]^{\!\!\!\!\phi_n}\ar@{->>}[d]&\StU(2n+2,\,R,\,\fp)\ar[r]\ar@{->>}[d]&\ldots\\
\ldots\ar@{^(->}[r]&\EU(2n,\,R,\,\fp)\ar@{^(->}[r]&\EU(2n+2,\,R,\,\fp)\ar@{^(->}[r]&\ldots
}
$$
\end{df}

\begin{rm*}
Lemma~\ref{cent} implies that
$$
\KU(R,\,\fp)\rightarrowtail\StU(R,\,\fp)\epi\EU(R,\,\fp)
$$
is a central extension.
\end{rm*}

\begin{lm}
\label{stisperfect}
An odd unitary Steinberg group $\st$ is perfect.
\end{lm}
\begin{hint}
Use the fact that $n\geq3$ and Relations R5 and R8.
\end{hint}

We will need the results of the following section to compute the Schur multiplier of $\st$.

\section{Main lemma}

\begin{ml}
Let $n$ be integer such that $n\geq 4$ or $n=\infty$, $\epsilon$ be a central extension of $\st$ such that property \dag\ holds:
\begin{equation*}
[\epsilon\inv X_{ij}(a),\,\epsilon\inv X_{kh}(b)]=1,
\eqno{(\dag)}
\end{equation*}
where $a$, $b$ are elements of $R$ and $i$, $j$, $k$ and $h$ are indices from $\Omega$ such that $\Card\{i,-i,j,-j,k,-k,h,-h\}=8$, i.e. any two of these four indices neither coincide nor have a zero sum. Then $\epsilon$ splits.
\end{ml}

In this section $n$ and $\epsilon$ will be always as in the Main lemma, $\Omega$ will denote $\{1,\ldots,n,-n,\ldots,-1\}$ for integer $n$ and $\{1,\ldots,n,\ldots,-n,\ldots,-1\}$ for $n=\infty$. 

The idea of the proof is to find elements $S_{ij}(a)\in\epsilon\inv X_{ij}(a)$ and $S_i(u,\,a)\in\epsilon\inv X_i(u,\,a)$ such that relations R0--R9 hold for these elements. It will follow from this fact that there is a homomorphism $\sigma$ from $\st$ to the group spanned on these elements sending generators to generators what will immediately imply that $\epsilon$ splits. Now we will give a detailed proof of the Main lemma, but indeed all lemmae proved in this section correspond to one of the relations R0--R9.

The following commutation identities will be essentially used throughout this section.

\begin{lm}
Let $G$ be a group, $x$, $y$, $z$, $y_1,\ldots,y_m$ be elements of $G$. For any $a$, $b\in G$ we denote $aba\inv$ by $^ab$ and left-normed commutator $aba\inv b\inv$ by $[a,\,b]$. Then straightforward calculation shows that
\setcounter{equation}{0}
\renewcommand{\theequation}{C\arabic{equation}}
\begin{align}
&[xy,\,z]=\,^x[y,\,z]\cdot[x,\,z],
\\
&[x,\,yz]=[x,\,y]\cdot\,^y[x,\,z],
\\
&[x,\,y_1\cdot\ldots\cdot y_m]=[x,\,y_1]\cdot\,^{y_1}[x,\,y_2]\cdot\,^{y_1y_2}[x,\,y_3]\cdot\ldots\cdot\,^{y_1\cdot\ldots\cdot y_{m-1}}[x,\,y_m],
\\
&[x,\,y]\cdot[x,\,z]=[x,\,yz]\cdot[y,\,[z,\,x]],
\\
&^y[x,\,[y^{-1},\,z]]\cdot\,^z[y,\,[z^{-1},\,x]]\cdot\,^x[z,\,[x^{-1},\,y]]=1
\\
&^z[y,\,[z^{-1},\,x]]=[^zy,\,[x,\,z]].
\end{align}
\end{lm}

The next lemma is a stronger version of the Property \dag.

\begin{lm}
\label{r3}
Let $i\in\Omega$, $j\in\Omega\setminus\{\pm i\}$, $k\in\Omega\setminus\{-i,j\}$, $h\in\Omega\setminus\{i,-j,\pm k\}$. Then for any $a$, $b\in R$
$$
[\epsilon\inv X_{ij}(a),\,\epsilon\inv X_{kh}(b)]=1.
$$
\end{lm}

\begin{proof}
If $\Card\{\pm i,\pm j,\pm k,\pm h\}\neq8$ then using the fact that $n\geq4$ we can fix $l\in\Omega\setminus\{\pm i,\pm j,\pm k,\pm h\}$ and $x\in\epsilon\inv X_{ij}(a)$, $y\in\epsilon\inv X_{kl}(b)$, $z\in\epsilon\inv X_{lh}(1)$ (for $n=\infty$ we should work in the $\StU(2m,\,R,\,\fp)$ with $m$ large enough). Relation~R5 implies that $[y,\,z]\in\epsilon\inv X_{kh}(b)$ and Relation~R3 implies that $[x,\,y]$, $[x,\,z]\in\Ker(\epsilon)\subseteq\Cent\big(\Dom(\epsilon)\big)$. Thus using Identity~C2 we have
$$
1=[x,\,y\inv y]=[x,\,y\inv]\cdot[x,\,y],
$$
i.e. $[x,\,y\inv]=[x,\,y]\inv$ (so it is central) and the same about $[x,\,z\inv]$. Now using Lemma~\ref{sct} and C3 we obtain that
$$
[\epsilon\inv X_{ij}(a),\,\epsilon\inv X_{kh}(b)]=[x,\,[y,\,z]]=[x,\,y]\cdot[x,\,z]\cdot[x,\,y\inv]\cdot[x,\,z\inv]=1.
$$
If $\Card\{\pm i,\pm j,\pm k,\pm h\}=8$ we can just use the Property \dag.
\end{proof}

\begin{rm*}
It is easy to see that if $n\geq5$ or $n=\infty$ then Property \dag\ holds for every central extension of $\st$. Indeed, if $n\geq5$ then we can fix $l\not\in\{\pm i,\pm j,\pm k,\pm h\}$ in the proof above even if $\Card\{\pm i,\pm j,\pm k,\pm h\}=8$.
\end{rm*}

\begin{lm}
\label{r4}
Let $i\in\Omega$, $j\in\Omega\setminus\{-i\}$, $k\in\Omega\setminus\{i,\pm j\}$. Then for any $\lambda\in\fp$, $a\in R$
$$
[\epsilon\inv X_i(\lambda),\,\epsilon\inv X_{jk}(a)]=1.
$$
\end{lm}

\begin{hint}
Like previous lemma.
\end{hint}

\begin{lm}
\label{r5}
Let $i$, $j$, $k$, $h$ be indices from $\Omega$ such that $\Card\{\pm i,\pm j,\pm k,\pm h\}=8$. Then for any $a$, $b\in R$
$$
[\epsilon\inv X_{ki}(a),\,\epsilon\inv X_{ih}(b)]=[\epsilon\inv X_{kj}(ab),\,\epsilon\inv X_{jh}(1)].
$$
\end{lm}

\begin{proof}
Fix $x\in\epsilon\inv X_{ki}(a)$, $y\in\epsilon\inv X_{ij}(-b)$, $z\in\epsilon\inv X_{jh}(1)$. By Lemma~\ref{r3} $[z\inv,\,x]=1$ and thus identity~C5 implies that
$$
^y[x,\,[y\inv,\,z]]=\,^x[[x\inv,\,y],\,z].
$$
But using R5 we have that $[y\inv,\,z]\in\epsilon\inv X_{ih}(b)$, $[x\inv,\,y]\in\epsilon\inv X_{kj}(ab)$ and  $[x,\,[y\inv,\,z]]$, $[[x\inv,\,y],\,z]\in\epsilon\inv X_{kh}(ab)$ and thus commute with $x$ and $y$ by Lemma~\ref{r3}.
\end{proof}

\begin{df}
For $a\in R$, $k$, $h\in\Omega$ such that $k\not\in\{\pm h\}$ we will denote the commutator $[\epsilon\inv X_{ki}(a),\,\epsilon\inv X_{ih}(1)]$ by $S_{kh}(a)$, where $i\in\Omega\setminus\{\pm k,\pm h\}$. This definition does not depend on the choice of $i$ by Lemma~\ref{r5}.
\end{df}

\begin{rm*}
Lemma~\ref{r5} implies that $[\epsilon\inv X_{ki}(a),\,\epsilon\inv X_{ih}(b)]=S_{kh}(ab)$.
\end{rm*}

We want to find $S_{kh}(a)\in\epsilon\inv X_{kh}(a)$ such that Relations R0--R9 would hold for them. In particular, Relation~R5 should hold but central trick implies that this relation is equivalent to the identity in the remark above. So it was natural to define right hand side of this identity as it's left hand side.

\begin{lm}
\label{r1}
For any $i\in\Omega$, $j\in\Omega\setminus\{\pm i\}$, $a$, $b\in R$
$$
S_{ij}(a)S_{ij}(b)=S_{ij}(a+b).
$$
\end{lm}

\begin{proof}
Fix $l\in\Omega\setminus\{\pm i,\pm j\}$, $x\in\epsilon\inv X_{il}(1)$, $y\in\epsilon\inv X_{lj}(a)$, $z\in\epsilon\inv X_{lj}(b)$. By Lemma~\ref{r3} $[y,\,[z,\,x]]=[\epsilon\inv X_{lj}(a),\,\epsilon\inv X_{ij}(-b)]=1$ and thus C4 implies that
$$
[x,\,y][x,\,z]=[x,\,yz].
$$
\end{proof}

\begin{rm*}
As we mentioned earlier, Lemma~\ref{r1} implies that $S_{ij}(0)=1$ and $S_{ij}(a)\inv=S_{ij}(-a)$.
\end{rm*}

\begin{lm}
\label{r0}
For any $i\in\Omega$, $j\in\Omega\setminus\{\pm i\}$, $a\in R$
$$
S_{ij}(a)=S_{-j,-i}(\varepsilon_{-j}\overline a\varepsilon_i).
$$
\end{lm}

\begin{proof}
Fix $l\in\Omega\setminus\{\pm i,\pm j\}$. Obviously $\varepsilon_l\,\varepsilon_{-l}=-{\overline1\,}\inv$ so
\begin{align*}
S_{ij}(a)=[\epsilon\inv X_{il}(a),\,\epsilon\inv X_{lj}(1)]=[\epsilon\inv X_{-l,-i}(\varepsilon_{-l}\overline a\varepsilon_i),\,\epsilon\inv X_{-j,-l}(\varepsilon_{-j}\overline1\varepsilon_l)]=&\\
={[\epsilon\inv X_{-j,-l}(\varepsilon_{-j}\overline1\varepsilon_l)),\,\epsilon\inv X_{-l,-i}(\varepsilon_{-l}\overline a\varepsilon_i)]}\inv={S_{-j,-i}(\varepsilon_{-j}\overline1\varepsilon_l\,\varepsilon_{-l}\overline a\varepsilon_i)}\inv=&\\
=S_{-j,-i}(-\varepsilon_{-j}\overline1\varepsilon_l\,\varepsilon_{-l}\overline a\varepsilon_i)=S_{-j,-i}(\varepsilon_{-j}\overline a\varepsilon_i).
\end{align*}
\end{proof}

\begin{lm}
\label{r6}
For any $i\in\Omega$, $j\in\Omega\setminus\{\pm i\}$, $(u,\,a)$, $(v,\,b)\in\mathfrak L$
$$
[\epsilon\inv X_i(u,\,a),\,\epsilon\inv X_j(v,\,b)]=S_{i,-j}(\varepsilon_iB(u,\,v)).
$$
\end{lm}

\begin{proof}
Fix $l\in\Omega\setminus\{\pm i,\pm j\}$, $x\in\epsilon\inv X_i(u,\,a)$, $y\in\epsilon\inv X_{-l}(-v,\,b)$, $z\in\epsilon\inv X_{l,-j}(1)$ (note that $(-v,\,b)=(v,\,b)\leftharpoonup(-1)\in\mathfrak L$). Using the fact that $[z\inv,\,x]=1$ (Lemma~\ref{r4}) and Identity~C5 we have 
$$
^x[[x\inv,\,y],\,z]=\,^y[x,\,[y\inv,\,z]].
$$
Now one can check that $[x\inv,\,y]\in\epsilon\inv X_{ik}(\varepsilon_iB(u,\,v))$ and $[y\inv,\,z]\in\epsilon\inv\big(X_j(v,\,b)\cdot X_{-l,-j}(-\varepsilon_{-l}\overline b)\big)$ (note that $(v,\,b)\in\fp\leq\fp_{\max}$ so $y\inv\in\epsilon\inv X_{-l}(v,\,-\overline b)$). C2 implies that $[x,\,[y\inv,\,z]]=[\epsilon\inv X_i(u,\,a),\,\epsilon\inv X_j(v,\,b)]\cdot1$. Now use that $S_{i,-j}(\varepsilon_iB(u,\,v))$ commutes with $x$ and $y$ (by Lemma~\ref{r4}).
\end{proof}

\begin{lm}
\label{r8}
For any $i$, $j$, $k\in\Omega$, such that $\Card\{\pm i,\pm j,\pm k\}=6$, $(u,\,a)\in\fp$, $b\in R$
\begin{multline*}
S_{i,-k}(\varepsilon_i\overline b\,{\overline1}\inv\overline ab)[\epsilon\inv X_i((u,\,-\overline a)\leftharpoonup b),\,\epsilon\inv X_{-i,-k}(1)]=\\=S_{j,-k}(\varepsilon_j\overline ab)[\epsilon\inv X_j(u,\,-\overline a),\,\epsilon\inv X_{-j,-k}(b)].
\end{multline*}
\end{lm}

\begin{proof}
Fix elements $x\in\epsilon\inv X_j(\dot-(-u,\,a))$, $y\in\epsilon\inv X_{-j,-i}(-b)$, $z\in\epsilon\inv X_{-i,-k}(1)$, $w\in\epsilon\inv X_{j,-i}(\varepsilon_jab)$. Using C1 we have
$
[[x\inv,\,y]w,\,z]=\,^{[x\inv,\,y]}[w,\,z]\cdot[[x\inv,\,y],\,z]
$
and using C5 and the fact $[z\inv,\,x]=1$ we have
$
^x[[x\inv,\,y],\,z]=\,^y[x,\,[y\inv,\,z]],
$
so
\begin{multline*}
[[x\inv,\,y]w,\,z]=\\=\,^{[x\inv,\,y]}[w,\,z]\cdot\,^{x\inv}[y,\,[x,\,[y\inv,\,z]]]\cdot[x\inv,\,[x,\,[y\inv,\,z]]]\cdot[x,\,[y\inv,\,z]].
\end{multline*}
One can check that \\
$
[x,\,[y\inv,\,z]]=[\epsilon\inv X_j(u,\,-\overline a),\,\epsilon\inv X_{-j,-k}(b)]\in\epsilon\inv\big(X_{j,-k}(-\varepsilon_j\overline ab)\cdot X_k((u,\,a)\leftharpoonup b)\big),
$ $[[x\inv,\,y]w,\,z]=[\epsilon\inv X_i((u,\,-\overline a)\leftharpoonup b),\,\epsilon\inv X_{-i,-k}(1)]$,\\ $[x\inv,\,[x,\,[y\inv,\,z]]]=S_{j,-k}(\varepsilon_jB(-u,\,ub))$\\ (use Lemmae~\ref{r4} and~\ref{r6}), $[y,\,[x,\,[y\inv,\,z]]]=S_{i,-k}(-\varepsilon_i\overline b\,{\overline1}\inv\overline ab)$ (use Lemmae~\ref{r4} and~\ref{r5}) and $[w,\,z]=S_{j,-k}(\varepsilon_jab)$. Use Lemmae~\ref{r3}, \ref{r4} and~\ref{r1} to finish the proof.
\end{proof}

\begin{df}
For $k\in\Omega$, $(u,\,a)\in\fp$ we will denote by $S_k(u,\,a)$ the element $S_{i,-k}(\varepsilon_i\overline a)\cdot[\epsilon\inv X_i(u,\,-\overline a),\,\epsilon\inv X_{-i,-k}(1)]$. This definition does not depend on the choice of $i$ by Lemma~\ref{r8}.
\end{df}

\begin{rm*}
Observe that by definition $S_k((u,a)\leftharpoonup b)$ is exactly 
$$
S_{i,-k}(\varepsilon_i\overline b\,{\overline1}\inv\overline ab)\cdot[\epsilon\inv X_i((u,\,-\overline a)\leftharpoonup b),\,\epsilon\inv X_{-i,-k}(1)].
$$ 
Thus, Lemma~\ref{r8} implies (changing $a$ by $-\overline a$ and $k$ by $-k$) that
$$
S_{jk}(\varepsilon_jab)S_{-k}((u,\,-\overline a)\leftharpoonup b)=[\epsilon\inv X_j(u,\,a),\,\epsilon\inv X_{-j,k}(b)].
$$
\end{rm*}

Again, we wanted to find $S_i(u,\,a)$ such that R0--R9 would hold for them, in particular, Relation~R8, i.e. precisely the identity above. So we defined left hand side of that identity as it's right hand side.

\begin{lm}
\label{r9}
For any $i\in\Omega$, $j\in\Omega\setminus\{\pm i\}$, $a\in R$
$$
[\epsilon\inv X_{ij}(a),\,\epsilon\inv X_{j,-i}(b)]=S_i(0,\,-\varepsilon_{-i}\overline1ab+\overline b\,{\overline1}\inv\overline a\varepsilon_i).
$$
\end{lm}

\begin{proof}
Fix $t\in\Omega\setminus\{\pm i,\pm j\}$, $x\in\epsilon\inv X_{j,-t}(b)$, $y\in\epsilon\inv X_{-j,-t}(-\varepsilon_{-j}\overline a\varepsilon_i)$, and $z\in\epsilon\inv X_{-t,-i}(1)$. Using C1 we have
$
^y[y\inv\cdot\,^zy,\,[x,\,z]]=\,^{yy\inv}[\,^zy,\,[x,\,z]]$ $\cdot\,^y[y\inv,\,[x,\,z]].
$
Thus C5 and C6 imply that
$$
^x[[x\inv,\,y],\,z]=\,^y[x,\,[y\inv,\,z]]\cdot\big(\,^y[[y\inv,\,z],\,[x,\,z]]\cdot\,^y[[x,\,z],\,y\inv]\big).
$$
One can obtain that $[x\inv,\,y]\in\epsilon\inv X_t(0,\,-\overline{(-\varepsilon_{-i}\overline1ab+\overline b\,{\overline1}\inv\overline a\varepsilon_i)})$ using R0 and R9. Relations~R5 and R0 imply that $[y\inv,\,z]\in\epsilon\inv X_{ij}(a)$ and $[x,\,z]\in\epsilon\inv X_{j,-i}(b)$. Thus we obtain that $[x,\,[y\inv,\,z]]=S_{t,-i}(-\varepsilon_{-t}\overline b\,{\overline1}\inv\overline a\varepsilon_i)$ and $[[x,\,z],\,y\inv]=S_{t,-i}(\varepsilon_{-t}\varepsilon_{-i}\overline1ab)$. Now use Lemmae~\ref{r3},~\ref{r4} and \ref{r1} to finish the proof.
\end{proof}

\begin{lm}
\label{r2}
For any $i\in\Omega$, $(u,\,a)$, $(v,\,b)\in\fp$
$$
S_{i}(u,\,a)S_{i}(v,\,b)=S_i((u,\,a)\dotplus(v,\,b)).
$$
\end{lm}

\begin{proof}
Fix $t\in\Omega\setminus\{\pm i\}$ and $x\in\epsilon\inv X_{-t,-i}(1)$, $y\in\epsilon\inv X_t(v,\,-\overline b)$, $z\in\epsilon\inv X_t(u,\,-\overline a)$. Identity C4 implies that
$$
[z,\,x][y,\,x]=[[z,\,x],\,y][yz,\,x].
$$
By C2 one has $[[z,\,x],\,y]=1\cdot S_{t,-i}(\varepsilon_t\overline{B(u,\,v)})$ (use R8 and Lemmae~\ref{r6} and \ref{r0}). Now~\ref{r8} and \ref{r1} finish the proof.
\end{proof}

\begin{lm}
\label{r7}
For any $i\in\Omega$, $(u,\,a)$, $(v,\,b)\in\fp$
$$
[S_i(u,\,a),\,S_i(v,\,b)]=S_i(0,\,B(u,\,v)-B(v,\,u)).
$$
\end{lm}

\begin{hint}
Use Lemma~\ref{r2}.
\end{hint}

\begin{proof}[Proof of the Main lemma]
Lemmae~7--16 imply that Relations R3, R4, R5, R1, R0, R6, R8, R9, R2 and R7 respectively hold for $S_{ij}(a)$ and $S_i(u,\,a)$. Thus, there is a group homomorphism $\sigma$ from $\st$ on the group spanned on these elements, such that $\sigma(X_{ij}(a))=S_{ij}(a)$ and $\sigma(X_i(u,\,a))=S_i(u,\,a)$ (for $n=\infty$ we should use here the universal property of the direct limit). One can see that $\epsilon\sigma=1_{\st}$, i.e $\epsilon$ splits.
\end{proof}

\section{Schur multiplier of unitary Steinberg group}

In this section we will obtain the main results of our paper.

\begin{tm}
Let $\st$ be an odd unitary Steinberg group, where $n\geq5$ or $n=\infty$. Then $\M(\st)=1$. 
\end{tm}

\begin{proof}
$\st$ is perfect group by Lemma~\ref{stisperfect} so there is a universal central extension $\pi:U\twoheadrightarrow\st$. As we mentioned in the previous section when $n\geq5$ property \dag\ from the Main lemma holds for every central extension of $\st$, in particular for $\pi$. Thus, $\pi$ is split extension. But its domain $U$ is perfect (see remark after the definition of the universal central extension) so that by Lemma~\ref{perfectsplit}, extension $\pi$ is in fact an isomorphism.
\end{proof}

\begin{tm}
Let $\pi:U\twoheadrightarrow\sta8$ be a universal central extension. Then Schur multiplier $\M(\sta8)$ coincides with the subgroup of $U$ generated by the elements $\{[\pi\inv X_{ij}(a),\,\pi\inv X_{kh}(b)]\mid i,j,k,h\in\Omega,\ \Card\{\pm i,\pm j,\pm k,\pm h\}=8,\ a,b\in R\}$.
\end{tm}

\begin{proof}
Denote by $M$ the subgroup generated by $\{[\pi\inv X_{ij}(a),\,\pi\inv X_{kh}(b)]\mid i,j,k,h\in\Omega,\ \Card\{\pm i,\pm j,\pm k,\pm h\}=8,\ a,b\in R\}$. It is contained in the $\Ker\pi\subseteq\Cent U$ so it is normal. One has $\pi(M)=1$, and thus $\pi$ induces the natural morphism $\varpi:U/M\twoheadrightarrow\sta8$. Obviously, $\varpi$ is a central extension, its domain is perfect and property \dag\ holds for $\varpi$. Thus, by Main lemma and Lemma~\ref{perfectsplit} $\sta8\cong U/M$.
\end{proof}

\begin{tm}
Let $n\geq5$ or $n=\infty$. Suppose that $$\KU(2n,\,R,\,\fp)\subseteq\Cent(\st)$$ $($it holds for example when $\KU(2n-2,\,R,\,\fp)\rightarrow\KU(2n,\,R,\,\fp)$ is surjective or $n=\infty$, see Lemma~$\ref{cent})$. Then lemma~$\ref{ccisuce}$ implies that
$$
\KU(2n,\,R,\,\fp)=\M(\EU(2n,\,R,\,\fp)).
$$
\end{tm}

\end{document}